\documentclass[12pt,a4paper]{article}
\usepackage{amsmath}
\usepackage{amsthm}
\usepackage{amssymb}
\usepackage{graphicx}
\usepackage{tabularx}
\usepackage{color}
\usepackage{setspace}
\usepackage{capt-of}
\usepackage{float}
\usepackage{amssymb,latexsym,amscd,graphicx}
\usepackage{color}
\newtheorem{cor}{Corollary}[section]
\newtheorem{thm}{Theorem}[section]

\theoremstyle{definition}
\newtheorem{defn}{Definition}[section]
\theoremstyle{remark}

\definecolor{orange}{rgb}{1,0.5,0}
\usepackage{graphicx}
\graphicspath{{converted_graphics/}}

\usepackage[top=2 cm,bottom=2 cm,left=1.5 cm,right=1.5 cm]{geometry}
\usepackage{multicol}
\usepackage{hyperref}
\usepackage{setspace}
\hypersetup{
	colorlinks,
	citecolor=blue,
	filecolor=black,
	linkcolor=black,
	urlcolor=black
}
\begin{document}
	
\title{Boundary values and zeros of Harmonic Product of Complex-valued Harmonic Functions in a simply connected bounded Domain}

	\author{ Ayantu Guteta Fite and Hunduma Legesse Geleta  }        
	
	\maketitle
	
 First author : Ayantu Guteta Fite \\
  Corresponding author: Hunduma Legesse Geleta\\
  
\textbf{Department of Mathematics, College of Natural and Computational Sciences, Addis Ababa University, Addis Ababa, Ethiopia }\\

\textbf{Email}: {ayantu.guteta@aau.edu.et} ( {first author})\\
 \textbf{Email}: { hunduma.legesse@aau.edu.et} ({Corresponding author})

\pagenumbering{arabic}
	

	\date{\today}          
	\maketitle
	\vspace{2mm}
\pagenumbering{arabic}
     
     \vspace{3mm} 
	
	\begin{abstract} 
	The product of two complex-valued harmonic function is not in general complex-valued harmonic function. In this paper we show that if a complex-valued harmonic function is the product of two complex-valued harmonic functions, then it is the difference of two squares, one is analytic and the other is co-analytic. As a result of this, if one of the factors of the product is known, then the other factor is expressed in terms of the known factor explicitly.  As an application of this we determine the boundary value of one of the factors and the product provided the boundary value of the other factor is known. It is shown that the boundary value of the product is a pure imaginary constant. Moreover, if such a product and factors are complex-valued harmonic polynomials, then the number of  zeros of the product is at most half of the square of its degree which is by half smaller than what is known in the literature about the maximum number of zeros of complex-valued harmonic polynomials. The feature of this paper is to explore multipliers for some subspace of complex-valued harmonic functions and determine some nontrivial invariant subspace. 
	\end{abstract}
	
	\maketitle
	
	\textbf{Keywords/phrases}: {Complex-valued Harmonic Functions; dilatation; Harmonic product; Boundary value; Zeros of polynomials.}
	
	\def\theequation{\thesection.\arabic{equation}}
	\section{Introduction}
	\setcounter{equation}{0}

In 1999  W. Hayman and et al \cite{WH} determined when the product of two non-constant, real-valued harmonic functions in a domain $\Omega$ in the complex plane is harmonic. It was asked in the  year 1998 by W. Hayman himself as cited in \cite{RM}.  Raymond Mortini in \cite{RM} also tried to answer the question of when the product of complex-valued harmonic functions, is harmonic. He uses what he called $\mathbb{C}$-harmonic conjugates. In this work we extend these works by using different approaches and explore more  conditions under which the product of complex-valued harmonic functions is complex-valued harmonic function and characterize in terms of dilatation and boundary values of the factors. As an auxiliary result we also determine the maximum number of zeros of the product provided the factors are complex-valued harmonic polynomials. We hope such topics become potential for the future research for instance in the study of multipliers and invariant subspace in the space of complex-valued harmonic functions.  \\
	
Complex-valued harmonic functions have been introduced in 1984 Clunie and Sheil-Small \cite{JT}. The family of complex-valued harmonic functions $f = u + iv$ defined in the unit disk $\mathbb{D} = \{z : |z| < 1\},$ where $u$ and $v$ are real harmonic functions in $\mathbb{D},$ but not necessarily harmonic conjugates. In any simply connected sub-domain $G \subset \mathbb{C}$ a complex-valued harmonic function $f$ can be decomposed as
 
		$$f(z) = h(z) + \overline{g(z)}$$
  where $g$ and $h$ are analytic \cite{duren2001univalent}. This family of complex-valued harmonic functions is a generalization of analytic mappings studied in geometric  function theory, and much research has been  done investigating the   properties of  these harmonic functions. For an  overview of the \bigskip topic, see Duren \cite{duren2004harmonic} and Dorff and Rolf \cite{dorff2012anamorphosis}. 

Since then, the theory of complex-valued harmonic functions and its subclass complex-valued harmonic univalent  functions attracted attention of complex analysts and the theory has rapidly developed and became an active branch of complex analysis for researchers\cite{dorff2012anamorphosis}.  Some areas of investigation that has  become of interest is the number and location of zeros of complex-valued harmonic polynomials; the boundary values and dilatation of complex-valued harmonic functions on bounded simply connected domains.\\

Many types of physical problems require solutions of Laplace's equation, and there exists a wide variety of solutions containing many different kinds of functions. However, a specific physical problem usually asks for a solution that is defined in a certain region and satisfies a given condition on the boundary of that region. The analysis of harmonic function is a natural extension of the study of analytic functions. Harmonic functions show up in many branches of physics where Laplace's equation appears, which include electricity magnetism and quantum mechanics. Therefore, mathematical research into these function is motivated not only by their mathematical relevance but also their applications to physical systems \cite{RS}.\\

Another area of investigation that has  become of interest is the boundary behavior of a complex-valued harmonic functions since they solve boundary value problems, known as the Dirichlet problem. It is the problem of finding a function that is continuous on the closed region, and is harmonic in the interior of the region. One version of  it can be  stated as follows: Given a prescribed function $\phi$ on the boundary of the region, interpreted as the distribution of heat on the boundary, find harmonic function u on the region such that $u = \phi$ on the boundary of the region.\\
 
Univalent harmonic mappings have attracted  the attention of complex analysts  after the appearance of the basic paper by Clunie and Sheil-Small~\cite{JT} in 1984. The work of these researchers and several others gave rise to several fascinating problems, conjunctures , and many tantalizing but perplexing questions. Though several researchers solved some of these problems  and conjunctures, yet many  perplexing questions are still unanswered and need to be investigated. For more information, concerning the determination of when the product of two non-constant complex-valued harmonic function is harmonic we refer to W. Hayman and et al \cite{WH} and  Raymond Mortini \cite{RM}; concerning the Boundary values we refer to Laugesen \cite{RS},Bshouty, D. and et al \cite{AD},  W. Hengartner and  G. Schober \cite{WG}; concerning the zeros of complex-valued harmonic polynomials we refer to  Bshouty ~\cite{bshouty2010problems}, Brilleslyper et al. \cite{brilleslyper2020zeros},  Kennedy \cite{kennedy1940bounds},  Dehmer \cite{dehmer2006location} and H.Legesse Geleta  and Alemu, O. A. \cite{geleta2022} on the number and location of the zeros of complex-valued harmonic polynomials in general and on the number and location of the zeros of trinomials in particular.\\

 Motivated by these works, we try to answer the following questions:\\

\begin{itemize}
\item[Problem 1.] When the product of two complex-valued harmonic functions is complex-valued harmonic function?
\item[Problem 2.]What is the relation among the boundary values of such product and its factors? 
\item[Problem 3.] If the factors are polynomials, then what is the maximum number of zeros of such product?\\
\end{itemize}

Throughout this article, $\mathbb{C}$, $\Omega$, $\partial\Omega,$   $\mathbb{D}$, $\partial\mathbb{D}$ and $\omega$  denote the complex plane, region in the complex plane, boundary of the region, the open unit disk, the unit circle and the dilatation respectively.\\

 Therefore, the purpose of this paper is to characterize factors and products of complex-valued harmonic functions, study their  boundary behaviors and  in particular study the zeros of such product provided it is the product of two polynomials. The feature of this paper is to explore multipliers for some subspace of complex-valued harmonic functions and determine some nontrivial invariant subspace.

Specifically it is shown that, if $f, F$ and $fF$ are complex-valued harmonic functions  on a bounded simply connected domain in the complex plane and $f(z)=h(z)+\overline{g(z)}, $ then
$F(z)= \alpha h(z)-\alpha \overline{g(z)}$ and $(fF)(z)= \alpha h^2(z)-\alpha \overline{g}^2(z),$ where $\alpha$ is nonzero real number.

More over if $f$ has the following property,

\[	
\begin{cases}
		\Delta f(z) = 0   & \text{if } z\in \mathbb{D} \\
		f(z)=\phi_1(z)+i\phi_2(z)   & \text{if } z \in \partial \mathbb{D}
	\end{cases}
	\]
then $F$ and $fF$ satisfy the following properties respectively, 

\[	
\begin{cases}
		\Delta F(z) = 0  & \text{if } z\in \mathbb{D} \\
		F(z)=\phi_2(z)+i\phi_1(z)   & \text{if } z \in \partial \mathbb{D}
	\end{cases}
	\]

	\[	
\begin{cases}
		\Delta (fF)(z) = 0  & \text{if } z\in \mathbb{D} \\
		(fF)(z)= i\alpha (\phi_1^2(z)+\phi_2^2(z))   & \text{if } z \in \partial \mathbb{D}.
	\end{cases}
	\]
It is also shown in particular, that if such product is the product of complex-valued harmonic polynomials, then  the number of zeros of the product $fF$ is half of the square of its degree.\\	
	
 This paper is organized  as follows: In section 2, we review the works of different scholars which may be used in the main results. In section 3, we state and prove our results. In section 4, we come up with conclusion.

   \section{Preliminaries}
   In this section we review some important concepts and results that we will use later on to prove the main results. \\
   
   \begin{defn}(\cite{RM})
   Every harmonic mapping $v$ on $\Omega$ making $u+iv$ holomorphic, is called a $\mathbb{C}$-harmonic conjugate of of $u.$
   \end{defn}
   \begin{thm}(as cited in \cite{RM})If $uv$ is harmonic on $\Omega,$ then$u$ has a harmonic conjugate $\tilde{u}$ on $\Omega$ and there are real constants $\alpha$ and $\beta$ such that $v=\alpha \tilde{u}+\beta. $

   \end{thm}
   \begin{thm}(\cite{RM})Let $u$ and $v$ be two non-constant complex-valued harmonic functions on $\Omega$. Suppose that $uv$ is harmonic. Then either both $u$ and $v$ are holomorphic or anti-holomorphic on $\Omega$ or $u$ has a normalized  $\mathbb{C}$-harmonic conjugate $\tilde{u}$ on $\Omega$ such that there are constants $\alpha$ and $\beta$ for which $ v=\alpha \tilde{u}+\beta.$ \\
   Conversely, let $u$ be a harmonic mapping, $\tilde{u}$ a normalized $\mathbb{C}$-harmonic conjugate of $u$ and let $ v=\alpha \tilde{u}+\beta.$ Then $v$ and $uv$ are harmonic.
   
     If $uv$ is harmonic on $\Omega,$ then$u$ has a harmonic conjugate $\tilde{u}$ on $\Omega$ and there are real constants $\alpha$ and $\beta$ such that $v=\alpha \tilde{u}+\beta. $

   \end{thm}
  \begin{defn}.
A complex valued harmonic functions $f(z)$ = $h(z)+\overline{g(z)}$ is said to be sense preserving at $z_{0}$ if $J_{f}(z_{0})$ $> 0$ and is sense reveserving at $z_{0}$ if $J_{f}(z_{0})$ $< 0$, where  $J_{f}(z_{0})$ is the jacobian of $f$ which is given by \\
   $$J_{f}(z_{0})= |f_{z}|^2-|f_{\overline{z}}|^2 = |h'|^2-|g'|^2$$

 Its dilatation $w$ is given by $\omega(z) =  \frac{g'(z)}{h'(z)}.$ 
 The dilatation may be interpreted geometrically to represent the "stretching" or "distortion " of $f$. Dilatation $\omega$ measures how far $f$ is from being conformal. The mapping $f$ is conformal if and only if $\omega = 0.$
\end{defn}

   The next theorem is from well known Analysis book by Walter Rudin, provides the solution of a boundary value problem-the Dirichlet problem: A continuous function $f$ is given on the boundary of the unit disk and it is required to find a harmonic function $F$ on tn the disk whose boundary values are $f$. 
    \begin{thm} (Theorem 11.9-Walter Rudin \cite{WR})
    Suppose $u$ is a continuous real function on the closed unit disk,     ~~$\overline{\mathbb{D}}$ and suppose $u$ is harmonic in $\mathbb{D}$. Then in the disk $u$ is the Poisson integral of its restriction to the boundary, and $u$ is the real part of the holomorphic function
     $$f(z)=\frac{1}{2\pi}\int_{-\pi}^{\pi}\frac{e^{it}+z}{e^{it}-z}u(e^{it})dt,  ~~~ z\in \mathbb{D}.$$  
    \end{thm}

     An important tool in finding  the solution of the Dirichlet Problem is the Poisson integral formula, which can be formulated as follows:   \\
 
 Let $\Omega$ be a region in the complex plane and  $f^*: \partial\Omega\rightarrow\mathbb{C}$ be continuous. Then there is a continuous function $f: \overline{\Omega}\rightarrow\mathbb{C}$ such that $f^*(z) = f(z)$ for all $z\in \partial\Omega$ and $f$ is harmonic in $\Omega.$ In particular, if $f^*(e^{i\theta})$ is a Lebesgue  integrable function on $\partial\mathbb{D},$  then the poisson integral $$f(z) = \frac{1}{2\pi}\int_{-\pi}^{\pi}P(r,\varphi-\theta)f^*(e^{i\varphi})d\varphi,~~ z=e^{i\theta}\in\partial\mathbb{D},$$ where $P(r,\theta)$ is the poisson kernel of $\mathbb{D}$,  is the solution for the Dirchlet problem; moreover, it is harmonic mapping of $\mathbb{D}$ ~\cite{bshouty2010problems}.\\
 
 It is well known that the Fundamental Theorem of Algebra stipulates that, a polynomial of degree $n>0$ in one variable has at most $n$ zeros counting multiplicities.  Note that the Fundamental Theorem of Algebra gives us an upper bound on the total number of roots of a polynomial.  A polynomial may not achieve the maximum allowable number of roots given by the Fundamental Theorem of Algebra, but this bound is sharp, in the sense that given a natural number $n$ there exist a polynomial of degree $n$ which has exactly $n$ zeros counting multiplicities.

   \begin{thm} (Wilmshurst \cite{WA})
   Let $h(z)$ and $g(z)$ be analytic polynomials of degree $n$ and $m$ respectively with $n > m.$ Then $f(z) = h(z) + \overline{g(z)},$ has at most $n^2$ zeros counting multiplicities. The case $n=m$ could have infinite number of  zeros.
   
   \end{thm}
   It was shown by Bshouty \textit{et al.} \cite{bshouty1995exact} that there exists  a complex-valued harmonic polynomial $ f = h + \overline{g},$ such that $h$  is an analytic polynomial of degree $n,$ $g$  is an analytic polynomial of degree $m \leq n$ and $f$ has exactly $n^2$ zeros counting with  multiplicities in the field of complex numbers,  $\mathbb{C}.$ \\ 
	\maketitle
\def\theequation{\thesection.\arabic{equation}}
\section{Main Results} 

	In this section we state and prove the main results of the paper. The first part of this section is results regarding products of complex-valued harmonic functions; the second part is results about the boundary values of the factors ans products; and the third part is an auxiliary result concerning the the maximum number of zeros of the product, provided the product and factors are complex-valued harmonic  polynomials.

	\setcounter{equation}{0}

\subsection{Product of harmonic functions}
	We know that the square of real harmonic function $u$ cannot be harmonic, unless $u$ is a constant. Of course the square of analytic function is analytic and hence harmonic as any analytic function is complex-valued harmonic function. Similarly the square of anti-analytic function is anti-analytic and hence complex-valued harmonic function. But for the complex-valued harmonic function which is neither analytic nor anti-analytic we have the following result
	\begin{thm}
	Suppose $f(z) = h(z) + \overline{g(z)},$ is harmonic on $\Omega.$ Then $f^2,$ is harmonic on $\Omega$ if and only if either $h$ or $g$ is a constant on $\Omega.$
	\end{thm}
	\begin{proof}
	If either $h$ or $g$ is a constant on $\Omega,$ then it easy to see that the square of analytic function or the square of co-analytic function is harmonic on $\Omega.$ Assume that the square of a complex-valued harmonic function is harmonic on $\Omega.$ Then $$f^2=(h + \overline{g})^2 = h^2 +\overline{g}^2+ 2h\overline{g} $$ is harmonic on $\Omega$ implies that, $$\frac{\partial^2 f^2}{\partial z\overline{z}} = 0.$$ From which we obtain, $$\frac{\partial^2 f^2}{\partial z\overline{z}} = 2\frac{\partial h}{\partial z}\frac{\partial \overline{g}}{\partial\overline{z} } =0, $$
	This on turn implies that either of the following two is satisfied on $\Omega,$ $$\frac{\partial h}{\partial z}= 0, ~~  ~~ \frac{\partial \overline{g}}{\partial\overline{z} } =0.$$ Hence either $h$ or $ \overline{g} $ is a constant $\Omega.$ Which is the desired result.
	
	\end{proof}
	
	\begin{thm}
	Suppose $f(z) = h(z) + \overline{g(z)},$ and $F(z) = H(z) + \overline{G(z)},$ are complex-valued  harmonic functions on $\Omega.$ Then $fF,$ is complex-valued harmonic function on $\Omega,$ if and only if  $k(z)= h(z)\overline{G}+H(z)\overline{g}$ is a constant.
	\end{thm}
	\begin{proof}
	If $k(z)$ is a constant then obviously $fF$ is harmonic. Now assume $fF$ is harmonic. Then either $k(z)$ is analytic or anti-analytic. This implies that either $k_{\overline{z}}(z)=0$ or $k_z(z)=0.$ This leads to, either $$ h(z)\overline{\frac{\partial}{\partial z} G}+H(z)\overline{\frac{\partial}{\partial z} g} =0$$ or $$\frac{\partial}{\partial z} h(z)\overline{G}+ \frac{\partial}{\partial z} H(z)\overline{g}=0$$ From which we obtain $$\frac{H(z)}{h(z)}=-\frac{\overline{\frac{\partial}{\partial z} G}}{\overline{\frac{\partial}{\partial z} g}}=\alpha$$ or $$ \frac{\frac{\partial}{\partial z} H(z)}{\frac{\partial}{\partial z} h(z)}= -\frac{\overline{G}}{\overline{g}}=\beta $$ For the ratios to be equal $\alpha$ and $\beta$ must be real non-zero constants. Hence we obtain the following, $$H(z)=\alpha h(z), ~~\overline{\frac{\partial}{\partial z} G}= -\alpha \overline{\frac{\partial}{\partial z} g},~~~~ \frac{\partial}{\partial z} H(z)= \beta \frac{\partial}{\partial z} h(z), ~~~~ \overline{G} = -\beta \overline{g} $$ This implies that $G(z)= -\alpha g(z)+ c_1,$ ~~  $H(z)= \alpha h(z),$ ~~$H(z)=\beta h(z) +c_2,$ and $G(z)= -\beta g(z),$ where $c_1$ and $c_2$ are constants.
	After some algebraic calculations we obtain $c_1=0=c_2$  and $\alpha = \beta$ which gives,
	$H(z)=\alpha h(z)$ and $\overline{G(z)} = -\alpha \overline{g(z)}  .$ Thus  $k(z)= h(z)\overline{G}+H(z)\overline{g}=   -\alpha h(z)\overline{g}+ \alpha h(z)\overline{g} = 0.$ This shows that $k(z)$ is a constant which is actually identically zero. Therefore, for the product $fF$ of two complex-valued harmonic functions $f$ and $F$ to be harmonic and if $f(z) = h(z) + \overline{g(z)},$ then $F(z)=\alpha h(z)+-\alpha \overline{g(z)}.$
	 
	\end{proof}
	
	\begin{cor}
	Let $f(z) = h(z) + \overline{g(z)}$ and $F(z)=\alpha h(z)+-\alpha \overline{g(z)}$ be complex valued harmonic functions on some region $\Omega$ in  $\mathbb{C}$ with dilatations, $w_f$ and $w_F$ respectively. Then if $\alpha$ is a non-zero real number, then $w_f = -w_F$.
	\end{cor}

	Now from these results one can conclude that if $f(z) = h(z) + \overline{g(z)},~ F(z),$ and $f(z)F(z)$ are complex-valued  harmonic functions on $\Omega,$ then for non-zero real constant $$F = \alpha h - \alpha \overline{g},$$ and $$fF = \alpha~ h^2 - \alpha~ \overline{g}^2.$$   Moreover,their dilatation $w_F$ of $F$ is  given by  $w_F = -w_f.$  
	
\section*{Observations:}\begin{itemize}
\item[1.] If $f, F$ and $fF$ are complex-valued harmonic functions, then $w_f = -w_F.$
\item[2.] If $f$ and $F$ are complex-valued harmonic functions and $w_F +w_f \neq 0,$  then $fF$ is not harmonic.
\item[3.] If $f, F$ and $fF$ are complex-valued harmonic functions, and $f$ is univalent, then $F$ is also univalent.

\item[4.] If $f$ and $fF$ are complex-valued harmonic functions, then $F$ need not be harmonic.
Example:$$f(z) = h(z) + \overline{g(z)} , (fF)(z)= h^3(z) + \overline{g(z)}^3$$  are both harmonic, but the following need not be harmonic$$F(z)= h^2(z) - (h\overline{g})(z) + \overline{g(z)}^2.$$ 
\item[5.] If $f$ and $F$ are complex-valued harmonic functions with $w_f = -w_F$, then $fF$ need not be harmonic.
Example: If $f(z)= h(z) + \overline{g(z)}, ~~ F(z) = \alpha h(z)- \overline{\alpha g},$ then we have $w_f = -w_F$ but $fF$ need not be harmonic unless $\alpha$ is a non-zero real constant.

\end{itemize}
	\subsection{Relation between Boundary values of Factors and the product} 
	
	\setcounter{equation}{0}
In this subsection	we prove results concerning boundary values of the factors and the product on the unit disk. 
	
	\begin{thm}
	Suppose $f(z), F(z)$ and  $(fF)(z)$ are complex-valued harmonic functions in the open unit disk and $f(z)= \phi_1(z) +i \phi_2(z)$ on the boundary of the unit disk. Then $F(z) = \alpha\phi_2(z) +i \alpha\phi_1(z) $ on the boundary of the unit disk, for some non-zero real constant $\alpha.$ 
	\end{thm}
	\begin{proof}
	Suppose $f(z)= u+iv =h(z) + \overline{g(z)} $  and satisfies 
	\[	
\begin{cases}
		\Delta f(z) = 0   & \text{if } z\in \mathbb{D} \\
		f(z)=\phi_1(z)+i\phi_2(z)   & \text{if } z \in \partial \mathbb{D}
	\end{cases} 	
\]
	 Then $u$  is the real parts of some analytic function say, $G_1$ and $v$ is the imaginary parts of some analytic function say, $G_2.$ Moreover, $G_1,u, G_2, v$ are given as follows:
	$$G_1(z)=\frac{1}{2\pi}\int_{-\pi}^{\pi}\frac{e^{it}+z}{e^{it}-z}\phi_1(e^{it})dt $$ 
Since $u$ is the real part of $G_1$, we have $$u = Re(G_1)= \frac{G_1+\overline{G_1}}{2} $$
	$$G_2(z)=\frac{1}{2\pi}\int_{-\pi}^{\pi}\frac{e^{it}+z}{e^{it}-z}\phi_2(e^{it})dt $$ 
 Since $v$ is the imaginary part of $G_2$, we have $$v = Im(G_2)= \frac{G_2-\overline{G_2}}{2i}. $$
 Thus the function $f(z)$ has the following form
 $$f(z)= h(z) + \overline{g(z)}=u+iv=\frac{G_1+\overline{G_1}}{2}+i \frac{G_2-\overline{G_2}}{2i}.$$
 From this we obtain $$h(z)= \frac{G_1+G_2}{2},~~~~~~g(z)= \frac{G_1-G_2}{2}$$ But then, $F(z)= \alpha h(z)-\alpha \overline{g(z)}$, implies that $$F(z)= \alpha \frac{G_1+G_2}{2}-\alpha\overline{\frac{G_1-G_2}{2}}=\alpha \frac{G_1-\overline{G_1}}{2}+\alpha \frac{G_2+\overline{G_2}}{2} $$ Therefore $F(z)$ can be rewritten in the following form
 $$F(z)=\alpha Re(G_2)+i\alpha Im(G_1)= U+iV$$
 From which we obtain,
 $$U= \alpha Re(G_2)  , ~~~~  V= \alpha Im(G_1).$$
 So U is the real part of (using the expressions of $G_2$) the following
 $$\alpha G_2(z)= \alpha \frac{1}{2\pi}\int_{-\pi}^{\pi}\frac{e^{it}+z}{e^{it}-z}\phi_2(e^{it})dt =\frac{1}{2\pi}\int_{-\pi}^{\pi}\frac{e^{it}+z}{e^{it}-z}\alpha \phi_2(e^{it})dt $$
  and  $V$ is the imaginary part of (using the expressions of $G_1$) the following,
  
  $$\alpha G_1(z)= \alpha \frac{1}{2\pi}\int_{-\pi}^{\pi}\frac{e^{it}+z}{e^{it}-z}\phi_1(e^{it})dt =\frac{1}{2\pi}\int_{-\pi}^{\pi}\frac{e^{it}+z}{e^{it}-z}\alpha \phi_1(e^{it})dt. $$
 Therefore, we obtain 
  \[	
\begin{cases}
		\Delta F(z) = 0  & \text{if } z\in \mathbb{D} \\
		F(z)=\phi_2(z)+i\phi_1(z)   & \text{if } z \in \partial \mathbb{D}
	\end{cases}
	\]
 which is the desired result.

	\end{proof}
	Observe that the real part of the boundary value of $f$ becomes the imaginary part of the boundary value of $F$ and the imaginary part of the boundary value of $f$ becomes the real part of the boundary value of $F.$
	\begin{cor}
	Suppose $f, F$ and the product $fF$ are complex-valued harmonic functions in the open unit disk and $f(z)= \phi_1(z) +i \phi_2(z)$ on the boundary of the unit disk. Then $(fF)(z)=i\alpha (\phi_1^2+\phi_2^2)$ on the boundary of the unit disk.
	\end{cor}
	\begin{proof}
 Since $f(z)=\phi_1(z) +i \phi_2(z)$ and $F(z)= \alpha \phi_2+i\alpha \phi_1$ on the boundary (as shown in Theorem 3.1), we have $$(fF)(z)=(\phi_1(z) +i \phi_2(z))(\alpha \phi_2+i\alpha \phi_1).$$
 From this we obtain $$(fF)(z)=i\alpha (\phi_1^2+\phi_2^2).$$ Therefore, we have 
 \[	
\begin{cases}
		\Delta (fF)(z) = 0  & \text{if } z\in \mathbb{D} \\
		(fF)(z)= i\alpha (\phi_1^2(z)+\phi_2^2(z))   & \text{if } z \in \partial \mathbb{D}.
	\end{cases}
	\]
	\end{proof}
\section*{Remark:}	Since $f(z)= \phi_1(z) +i \phi_2(z)$ on the boundary of the unit disk we must have $\phi_1^2(z)+\phi_2^2(z)=1.$ Thus $fF$ must be a constant on the boundary which is either $i$ or $-i$, depending on the value of $\alpha,$ as the value of $\alpha$ is either 1 or -1.
	
	\subsection{Maximum Number of Zeros of the Product}
It has been shown that, if $f, F$ and $fF$ are complex-valued harmonic functions  on a domain in the complex plane and $f(z)=h(z)+\overline{g(z)}, $ then
$F(z)= \alpha h(z)-\alpha \overline{g(z)}$ and $(fF)(z)= \alpha h^2(z)-\alpha \overline{g}^2(z),$ where $\alpha$ is non-zero real number. Thus if we know the maximum number of zeros of $f$ in the complex plane, then the number of zeros of the product $fF$ is easily determined. In particular, if $h$ and $g$ are complex-valued analytic polynomials with degree $n$ and $m$ respectively satisfying $n > m,$ then $f$ is a complex-valued harmonic polynomial with degree $n$ and $fF$ is also a complex-valued harmonic polynomial with degree $2n.$
\begin{thm}
Suppose $fF$ and all its factors $f$ and $F$ are complex-valued harmonic polynomials and assume $f$ is a polynomial of degree $n$. Then $fF$ is a complex-valued harmonic polynomial with degree $2n$ and has at most $2n^2$ zeros in the complex plane.

\end{thm}
\begin{proof}
The zeros of $fF$ is the sum of the zeros of $f$ and $F$ counting with multiplicities. Which is equal to at most $n^2+n^2=2n^2,$ since $f(z)=h(z)+\overline{g(z)} $ has at most $n^2$ zeros and $F(z)= \alpha h(z)-\alpha \overline{g(z)}$ has also at most $n^2$ zeros.
\end{proof}		
Observe that the degree of $fF$ is the sum of the degrees of $f$ and $F$ which is $2n,$ and so $fF$ has at most $(2n)^2=4n^2$ zeros, in general by Wilmshurst \cite{WA}, but here it is less than half of what is known. 
\section{conclusion}
In this paper we have shown that if a complex-valued harmonic function is the product of two complex-valued harmonic functions, then it is the difference of two squares, one is analytic and the other is co-analytic. It is also shown that if  the two factors are complex-valued harmonic functions, then their dilatation differ only by sign, so that  if one of the factors is given the other factor is completely determined.  In addition to this, we determined the boundary values of one of the factors and the product provided the boundary value of the other factor is known. In particular the boundary value of the product is a constant. Moreover, as an auxiliary result we determined the maximum number of zeros of the product provided all the factors complex-valued harmonic polynomials. \\

\end{document}